\documentclass[11pt]{article}

\usepackage{amsmath}
\usepackage{amssymb}
\usepackage{amsthm}

\newtheorem{thm}{Theorem}

\newtheorem*{conj*}{Conjecture}

\newtheorem{cor}{Corollary}

\title{The Diophantine equation $b (b+1) (b+2) = t a (a + 1) (a + 2)$ and gap principle}
\author{Tsz Ho Chan}
\date{}

\begin{document}
\maketitle

\begin{abstract}
In this article, we are interested in whether a product of three consecutive integers $a (a+1) (a+2)$ divides another such product $b (b+1) (b+2)$. If this happens, we prove that there is some gaps between them, namely $b \gg \frac{a (\log a)^{1/6}}{(\log \log a)^{1/3}}$. We also consider other polynomial sequences such as $a^2 (a^2 + l)$ dividing $b^2 (b^2 + l)$ for some fixed integer $l$. Our method is based on effective Liouville-Baker-Feldman theorem.
\end{abstract}

\section{Introduction and main results}

Consider products of consecutive positive integers $1 \cdot 2 = 2, \, 2 \cdot 3 = 6, \, 3 \cdot 4 = 12, \ldots$. Observe that $3 \cdot 4 = 2 (2 \cdot 3)$. So, one may ask if there are infinitely many integers $1 \le a < b$ such that $b (b+1) = 2 a (a + 1)$. Multiplying by $4$ and completing the squares, we get
\begin{equation} \label{eqpell}
(4b^2 + 4b) = 2 (4a^2 + 4a) \; \; \text{ or } \; \; (2b + 1)^2 - 2(2a + 1)^2 = -1.
\end{equation}
By the theory of Pell equation, all the positive integer solutions of $x^2 - 2y^2 = -1$ are generated by $x_n + y_n \sqrt{2} = (1 + \sqrt{2})^{2n - 1}$. One can prove by induction that all $x_n$'s and $y_n$'s are odd numbers. Therefore, we have infinitely many $a$ and $b$ satisfying \eqref{eqpell}, namely
\[
a = \frac{y_n - 1}{2} \; \; \text{ and } \; \; b = \frac{x_n - 1}{2}.
\] 
Hence, it can occur infinitely often that products of consecutive integers share the same prime factors, and one such product is the double of another one.

\bigskip

Next, it is natural to consider products of three consecutive positive integers $1 \cdot 2 \cdot 3 = 6, \, 2 \cdot 3 \cdot 4 = 24, \, 3 \cdot 4 \cdot 5 = 60, \ldots$. Note that $4 \cdot 5 \cdot 6 = 2 (3 \cdot 4 \cdot 5)$. Again, one may ask if there are infinitely many positive integers $a < b$ such that $b (b+1) (b+2) = 2 a (a + 1) (a + 2)$. It turns out that the answer is negative in this case. We deduce this by connecting the question to rational approximation of algebraic numbers and applying effective Liouville-Baker-Feldman theorem in the areas of Diophantine approximation and transcendental numer theory.
\begin{thm} \label{thm1}
The Diophantine equation $b (b+1) (b+2) = 2 a (a + 1) (a + 2)$ has $a = 3$ and $b = 4$ as its only positive integer solution.
\end{thm}
It is worthwhile to note that various people have studied similar Diophantine equations as above. For example,
\begin{align*}
\text{Mordell 1963 \cite{m:1963}: } & \; \; \; b (b + 1) = a (a + 1) (a + 2), \\
\text{Cohn 1971 \cite{c:1971}: } & \; \; \; b (b + 1) (b + 2) (b + 3) = 2 a (a + 1) (a + 2) (a + 3), \\
\text{Boyd, Kisilevsky 1972 \cite{bk:1972}: } & \; \; \; b (b + 1) (b + 2) = a (a + 1) (a + 2) (a + 3), \\
\text{Saradha, Shorey, Tijdeman 1995 \cite{sst:1995}: } & \; \; \; b (b + d_1) \cdots (b + (k-1) d_1) = a (a + d_2) \cdots (a + (k -1) d_2), \\
\text{Beukers, Shorey, Tijdeman 1999 \cite{bst:1999}: } & \; \; \; b (b + d_1) \cdots (b + (m-1) d_1) = a (a + d_2) \cdots (a + (n-1) d_2), \\
\text{Rakaczki 2003 \cite{r:2003}: } & \; \; \; b (b + 1) \cdots (b + (m-1)) = \lambda a (a + 1) \cdots (a + (n-1)) + l,
\end{align*} 
to name a few. Strengthening Theorem \ref{thm1}, one can prove some sort of gap between divisible terms of products of three consecutive integers as follows, but let us introduce some notation first.

\bigskip

{\bf Notation } The symbol $a \mid b$ means that $a$ divides $b$. The symbols $f(x) \ll g(x)$ and $g(x) \gg f(x)$ are all equivalent to $|f(x)| \leq C g(x)$ for some constant $C > 0$. Finally, $f(x) = O_{\lambda_1, \ldots, \lambda_r} (g(x))$, $f(x) \ll_{\lambda_1, \ldots, \lambda_r} g(x)$ or $g(x) \gg_{\lambda_1, \ldots, \lambda_r} f(x)$ mean that the implicit constant $C$ may depend on certain parameters $\lambda_1, \ldots, \lambda_r$.

\begin{thm} \label{thm2}
Given a positive integer $l$. Suppose $a (a+l) (a+2l) \mid b (b+l) (b+2l)$ for some sufficiently large positive integers $a < b$ (in terms of $l$). Then
\[
b \gg_l \frac{a (\log a)^{1/6}}{(\log \log a)^{1/3}}.
\]
\end{thm}
Thus, the polynomials $x (x+l) (x+2l)$ have gap order $1$. Here a polynomial $p(x)$ with integer coefficients has {\it gap order $1$} or satisfies {\it the gap principle of order $1$} if $b / a \rightarrow \infty$ whenever $p(a) \mid p(b)$ for some positive integers $a < b$ with $a \rightarrow \infty$. The study of such gap principle was started in \cite{c:2018}, \cite{ccl:2018} and \cite{c:2020}. It was later completed by Choi, Lam and Tarnu \cite{clt:2021} who showed that $p(x)$ has gap order $1$ if and only if $p(x)$ is not an integer multiple of some power of linear or quadratic polynomial. The interested readers can refer to \cite{ccl:2018} and \cite{clt:2021} for more general gap prinicple of higher order. As an immediate consequence of the above, we have
\begin{cor}
Given integers $l \ge 1$ and $t \ge 2$, the Diophantine equation
\begin{equation} \label{keyD}
b (b + l) (b + 2l) = t a (a + l) (a + 2l)
\end{equation}
has only finitely many integer solutions in $a$ and $b$.
\end{cor}
\begin{proof}
From Theorem \ref{thm2} and \eqref{keyD}, one has $\frac{(\log a)^{1/6}}{(\log \log a)^{1/3}} \ll_l b / a \ll_l \sqrt[3]{t}$. This shows that $a$ is bounded in terms of $t$ and $l$. Hence, there are only finitely many positive integer solutions in $a$. If $a < -2l$, one simply switch variables $a \mapsto -a$ and $b \mapsto -b$ and divide both sides of \eqref{keyD} by $-1$ to make everything positive again. Finally, there are only finitely many possibilities from $-2l \le a \le 0$. {\it Note:} The corollary also follows immediately from Theorems 2.1 and 2.2 of \cite{bst:1999} together with Faltings' theorem \cite{f:1983}.
\end{proof}
Applying similar ideas as Theorems \ref{thm1} and \ref{thm2}, one can improve and generalize gap results on the polynomial $x^2 (x^2 + 1)$.
\begin{thm} \label{thm3}
Given a non-zero integer $l$. Suppose $a^2 (a^2 + l) \mid b^2 (b^2 + l)$ for some positive integers $3 \le a < b$. Then
\[
b \gg_{l} \frac{a (\log a)^{1/12}}{(\log \log a)^{1/2}}.
\]
\end{thm}
One can also obtain a stronger gap result conditionally. Let us recall the famous $abc$-conjecture.
\begin{conj*}[$abc$-conjecture]
For every $\epsilon > 0$, there exists a constant $C_\epsilon$ such that for all triples $(a,b,c)$ of co-prime positive integers, with $a + b = c$, we have
\[
c < C_\epsilon \kappa(a b c)^{1 + \epsilon}
\]
with $\kappa(n) := \prod_{p \mid n} p$ being the squarefree kernel of $n$ where the product runs over prime factors $p$.
\end{conj*}
\begin{thm} \label{thm4}
Given a non-zero integer $l$ and any real number $0 < \epsilon < 2/5$. Suppose $a^2 (a^2 + l) \mid b^2 (b^2 + l)$ for some positive integers $1 \le a < b$. Under the $abc$-conjecture, we have
\[
b \gg_{\epsilon, l} a^{8 / 7 - \epsilon}.
\]
\end{thm}
Previously, only the case $l = 1$ was studied via hyperelliptic curves. The best results \cite{c:2020}, \cite{c:2021} were $b \gg \frac{a (\log a)^{1/960}}{(\log \log a)^{1/10}}$ unconditionally, and $b \gg_{\epsilon} a^{15/14 - \epsilon}$ conditionally on the $abc$-conjecture. We leave it as an exercise for the readers to derive a stronger conditional result for Theorem \ref{thm2}.

\bigskip

The paper is organized as follows. First, we will transform the Diophantine equation \eqref{keyD} to a rational approximation problem and prove Theorem \ref{thm1}. Then we recall some tools from Diophantine approximation and algebraic number theory, and apply them to prove Theorem \ref{thm2}. Finally, we will imitate the above techniques to prove Theorems \ref{thm3} and \ref{thm4}.

\bigskip

{\bf Acknowledgement } The author would like to thank the anonymous referee for helpful suggestions that greatly enhance the presentation of the paper.

\section{Initial transformation}

In this section, we relate the Diophantine equation $b (b + l) (b + 2l) = t a (a + l) (a + 2l)$ to rational approximation. First, we make a change of variables $x = a + l$ and $y = b + l$. Then, the equation becomes $y (y^2 - l^2) = t x (x^2 - l^2)$. Let $D = \gcd(x, y)$, $x = D u$ and $y = D v$ with relatively prime integers $v > u > 0$. Then, we have
\[
v (D^2 v^2 - l^2) = t u (D^2 u^2 - l^2) \; \; \text{ or } \; \; D^2 (v^3 - t u^3) = (v - t u) l^2.
\]
In particular, $D^2 \mid (v - t u) l^2$ and we can rewrite the above as
\begin{equation} \label{eqcube}
v^3 - t u^3 = -s \; \; \text{ where } \; \; s = \frac{(t u - v) l^2}{D^2}.
\end{equation}
Now, observe that
\begin{equation} \label{div}
s \mid \gcd((t u - v) l^2, v^3 - t u^3) \mid \gcd((t^3 u^3 - v^3) l^2, v^3 - t u^3) \mid \gcd((t^3 - t) l^2 u^3, v^3 - t u^3) \mid (t - 1) t (t + 1) l^2
\end{equation}
as $v^3 - t u^3$ and $u^3$ are relatively prime. Thus, $|s| \le t^3 l^2$ and it follows from \eqref{eqcube} and \eqref{div} that
\begin{equation} \label{dioph}
\Big| t - \frac{v^3}{u^3} \Big| \le \frac{t^3 l^2}{u^3} \; \; \text{ or } \; \; \Big| \sqrt[3]{t} - \frac{v}{u} \Big| \le \frac{t^{7/3} l^2}{u^3}
\end{equation}
which becomes a rational approximation question.

\section{Proof of Theorem \ref{thm1}}

\begin{proof}
In this section, we consider the special case $b (b + 1) (b + 2) = 2 a (a + 1) (a + 2)$ where $t = 2$ and $l = 1$ in \eqref{keyD}. By \eqref{eqcube}, this is equivalent to solving
\begin{equation} \label{special}
v^3 - 2 u^3 = - \frac{2 u - v}{D^2} \; \text{ where } \; x = D u = a + 1 \text{ and } y = D v = b + 1.
\end{equation}

When $u = 1$, \eqref{special} becomes $v^3 - 2 = \frac{v - 2}{D^2} \le v - 2$ which is impossible as $v > u = 1$.

\bigskip

When $u = 2$, \eqref{special} becomes $v^3 - 16 = \frac{v - 4}{D^2}$. One can easily see that $v = 3$ or $4$ are not possible. If $v \ge 5$, then $v^3 - 16 \le v - 4$ or $v^3 \le v + 12$ are also impossible.

\bigskip

When $u = 3$, \eqref{special} becomes $v^3 - 54 = \frac{v - 6}{D^2}$. One can easily see that $v = 4, 5$,or $6$ are not possible. If $v \ge 7$, then $v^3 - 54 \le v - 6$ or $v^3 \le v + 48$ are also impossible.

\bigskip

When $u = 4$, \eqref{special} becomes $v^3 - 128 = \frac{v - 8}{D^2}$. One has $v = 5$ and $D = 1$ as solutions. Translating back to $a$ and $b$, we have $a = Du - 1 = 3$ and $b = Dv - 1 = 4$ which corresponds to $4 \cdot 5 \cdot 6 = 2 (3 \cdot 4 \cdot 5)$. Next, one can easily see that $v = 6, 7$ or $8$ are not possible. If $v \ge 9$, then $v^3 - 128 \le v - 8$ or $v^3 \le v + 120$ is also impossible.

\bigskip

From now on, we assume $u \ge 5$. By \eqref{dioph} with $t = 2$ and $l = 1$, any solution to \eqref{special} satisfies
\[
\Big| \sqrt[3]{2} - \frac{v}{u} \Big| \le \frac{6}{u^3},
\]
which implies $\frac{v}{u} \ge \sqrt[3]{2} - 6/125 \ge 1.21$. Putting this into \eqref{dioph}, we have
\[
\Big| 2 - \frac{v^3}{u^3} \Big| = \Big| \sqrt[3]{2} - \frac{v}{u} \Big| \cdot \Big| \sqrt[3]{4} + \sqrt[3]{2} \frac{v}{u} + \frac{v^2}{u^2} \Big| \le \frac{8}{u^3}
\]
which gives
\begin{equation} \label{eqroot}
\Big| \sqrt[3]{2} - \frac{v}{u} \Big| \le \frac{8}{(\sqrt[3]{4} + 1.21 \sqrt[3]{2} + 1.21^2) u^3} < \frac{1.75}{u^3}.
\end{equation}
We apply a result of Bennett \cite{b:1997}:
\[
\Big| \sqrt[3]{2} - \frac{v}{u} \Big| > \frac{1}{4 u^{2.45}}
\]
and get $u^{0.55} < 4 \times 1.75$ or $u \le 35$. Then, one can use a simple computer program to verify that $u^2 |u \sqrt[2]{2} - v| > 5$ for all $5 \le u \le 35$. Hence, there is no solution to \eqref{eqroot} for $u \ge 5$ which gives Theorem \ref{thm1}.
\end{proof}

\section{Tools from Diophantine approximation and algebraic number theory}

Suppose $\alpha$ is a real algebraic number of degree $n$ over $\mathbb{Q}$. Then Liouville's theorem states that there exists some constant $c_\alpha > 0$ such that
\[
\Big| \alpha - \frac{y}{x} \Big| \ge \frac{c_\alpha}{x^n}
\]
for all integers $x, y$ with $x \neq 0$. The exponent $n$ was reduced to $2 + \epsilon$ by a series of works by Thue, Siegel, Dyson, and Roth. However, these are ineffective results. Using linear forms of logarithms, Baker \cite{b:1967} and Feldman \cite{f:1971} subsequently derived effective improvements of the form
\begin{equation} \label{BF}
\Big| \alpha - \frac{y}{x} \Big| \ge \frac{c_\alpha}{x^{n - \tau_\alpha}}
\end{equation}
for some $c_\alpha, \tau_\alpha > 0$. We will use the following explicit result of Bugeaud \cite{b:1998}: Let $A \ge e$ be an upper bound for the height of $\alpha$, and $R$ be the regulator of $\mathbb{Q}(\alpha)$. Then \eqref{BF} is true with
\begin{equation} \label{Bu}
c_\alpha = A^{-n^2} \exp( -10^{27n} n^{14n} R), \; \; \text{ and } \; \; \tau_\alpha = (10^{26 n} n^{14 n} R)^{-1}.
\end{equation}

Thus, we need information about the regulator of algebraic number fields. Given an algebraic number field $K$ over $\mathbb{Q}$ of degree $n$ with ring of integers $\mathcal{O}_K$, let $\mathcal{O}_K^*$ be the group of units of $\mathcal{O}_K$ and
\[
g_K := 2^{r_1} h_K R_K / w_K
\]
where $r_1$ is the number of real embeddings of $K$, $h_K$ is the class number of $\mathcal{O}_K$, $R_K$ is the regulator of $K$, and $w_K$ is the order of the torsion subgroup of $\mathcal{O}_K^*$. Siegel \cite{s:1969} gave an effective version of Landau's bound \cite{l:1918} on $g_K$:
\begin{equation} \label{gK}
g_K < 4 \Bigl( \frac{2}{n-1} \Big)^{n-1} \sqrt{|d_K|} \log^{n-1} |d_K|
\end{equation}
where $d_K$ is the discriminant of $K$. For pure cubic fields, Barrucand, Loxton and Williams \cite{blw:1987} proved the following better bound:
\begin{equation} \label{blw}
h_K R_K < \frac{1}{6 \sqrt{3}} \sqrt{|d_K|} \log |d_K|.
\end{equation}
When applying the above results, our algebraic number $\alpha$ is of the form $\sqrt[3]{m}$ or $\sqrt[4]{m}$ for some integer $m > 1$. Hence, we need information on the discriminant of pure cubic and quartic fields.

\bigskip

For pure cubic fields $\mathbb{Q}(\sqrt[3]{m})$, Dedekind partitioned them into Type I or Type II depending on whether $3$ is wildly ramified or is tamely ramified. Recently, Harron \cite{h:2017} determined the shapes of pure cubic fields nicely, and we quote his Lemma 2.1 here. Suppose $m = a b^2$ where $a$ and $b$ are relatively prime, squarefree, positive integers. For $K = \mathbb{Q}(\sqrt[3]{m})$,
\[
d_K = \left\{ \begin{array}{ll}
-3^3 a^2 b^2, & \text{ if $K$ is of Type I }, \\
-3 a^2 b^2, & \text{ if $K$ is of Type II }.
\end{array} \right.
\]
Hence, by \eqref{blw},
\begin{equation} \label{dR3}
|d_K| \le 27 m^2 \; \; \text{ and } \; \; R_K < m \log 6m.
\end{equation}

For pure quartic fields, we quote the following result of Funakura \cite[Corollary 1]{f:1984}: Suppose $m = a b^2 c^3$ where (i) $a \neq 1$, $b$ and $c$ are pairwise relatively prime squarefree numbers; (ii) $b$ and $c$ are positive; (iii) $|a| \ge c$ if $a$ is odd; (iv) $c$ is odd; (v) $m \neq -4$. For $K = \mathbb{Q}(\sqrt[4]{m})$,
\begin{equation} \label{Fu}
d_K = \left\{ \begin{array}{ll}
-2^2 a^3 b^2 c^3, & \text{ if } m \equiv 1 \pmod{8} \; \text{ or } \; 28 \pmod{32}, \\
-2^4 a^3 b^2 c^3, & \text{ if } m \equiv 4 \pmod{16}, \; 5 \pmod{8} \; \text{ or } \; 12 \pmod{32}, \\
-2^8 a^3 b^2 c^3, & \text{ if } m \equiv 2 \pmod{4} \; \text{ or } \; 3 \pmod{4}.
\end{array} \right.
\end{equation}
By unique prime factorization, we can write $m = u v^2 w^3 s^4$ for some relatively prime squarefree numbers $u, v, w > 0$, and some integer $s > 0$. Clearly, $\mathbb{Q}(\sqrt[4]{m}) = \mathbb{Q}(\sqrt[4]{u v^2 w^3})$. 

If $u > 1$ is even, then $w$ must be odd and we can apply \eqref{Fu} with $m = u v^2 w^3$. Hence, $|d_K| \le 256 m^3$. 

If $u = 1$, then we note that $\mathbb{Q}(\sqrt[4]{t}) = \mathbb{Q}(\sqrt[4]{t^3}) = \mathbb{Q}(\sqrt[4]{w v^2 u^3})$. Hence, we can apply \eqref{Fu} with $m = w v^2 u^3$ and get $|d_K| \le 256 m^3$.

If $w$ is even, then $u$ must be odd and $\mathbb{Q}(\sqrt[4]{t}) = \mathbb{Q}(\sqrt[4]{t^3}) = \mathbb{Q}(\sqrt[4]{w v^2 u^3})$. Hence, we can apply \eqref{Fu} with $m = w v^2 u^3$ and get $|d_K| \le 256 m^3$.

If $u > 1$ is odd and $w$ is odd, then we have either $u \ge w$ or $u < w$. In the former case, we can simply apply \eqref{Fu} with $m = u v^2 w^3$ and get $|d_K| \le 256 m^3$. In the latter case, we use $\mathbb{Q}(\sqrt[4]{m}) = \mathbb{Q}(\sqrt[4]{m^3}) = \mathbb{Q}(\sqrt[4]{w v^2 u^3})$ and apply \eqref{Fu} with $m = w v^2 u^3$ and still get $|d_K| \le 256 m^3$.

\bigskip

In summary,
\begin{equation} \label{dR4}
|d_K| \le 256 m^3 \; \; \text{ and } \; \; R_K < 15 m^{3/2} \log^3 8 m
\end{equation}
by \eqref{gK} as real quartic field $K$ has $\pm 1$ as its only torsion elements in $\mathcal{O}_K^*$ and $r_1 \ge 2$ ($\sqrt[4]{m} \mapsto \sqrt[4]{m}$ and $\sqrt[4]{m} \mapsto - \sqrt[4]{m}$).

\section{Proof of Theorem \ref{thm2}}

\begin{proof}
Consider the equation $b (b+l) (b+2l) = t a (a+l) (a+2l)$ for some fixed integer $t > 1$. We may assume that $t \le \log^2 a$ for otherwise the theorem is true as $b^3 \gg_l t a^3$. Combining \eqref{dioph} with \eqref{BF}, \eqref{Bu} and \eqref{dR3}, we have
\[
\frac{t^{-9} \exp( - 10^{102} t \log 6t )}{u^{3 - 1 / (10^{99} t \log 6t)}} < \Big| \sqrt[3]{t} - \frac{v}{u} \Big| \le \frac{t^{7/3} l^2}{u^3}.
\]
Therefore, $u^{1 / (10^{99} t \log 6t)} < l^2 t^{34/3} \exp( 10^{102} t \log 6t )$ which implies
\[
\log u \le 10^{201} ( t^2 \log^2 6t + t \log^2 6t + (2 \log l) t \log 6t ) \le 10^{202} (1 + \log l) t^2 \log^2 6t
\]
or
\[
t \ge \frac{\sqrt{\log u}}{10^{102} \sqrt{1 + \log l} \cdot \log \log u}
\]
for $u$ sufficiently large (in terms of $l$). Note that $s > 0$ for otherwise \eqref{eqcube} gives $t u \le v$ which implies $(v - t u) l^2 \ge (t^3 - t) u^3$ and $v \ge \frac{t^3 u^3}{2 l^2}$. This would contradict \eqref{dioph}. Since $D^2 \mid (v - t u) l^2$ and $s > 0$, we have $D^2 \le t u l^2$ or $D \le \sqrt{t} l \sqrt{u}$. Then, as $t \le \log^2 a$,
\[
\log D u \le \log (\sqrt{t} l u^{3/2}) \; \text{ or } \; \log u \ge \frac{2}{3} \log D u - \log \log D u - \log l \ge \frac{1}{2} \log D u
\]
for sufficiently large $a = D u$ (in terms of $l$). Hence,
\[
t \ge \frac{\sqrt{\log a}}{10^{103} \sqrt{1 + \log l} \cdot \log \log a}
\]
which gives Theorem \ref{thm2} as $b^3 \gg_l t a^3$.
\end{proof}

\section{Proof of Theorem \ref{thm3}}

\begin{proof}
Since $a^2 (a^2 + l)$ divides $b^2 (b^2 + l)$, we have
\[
t a^2 (a^2 + l) = b^2 (b^2 + l)
\]
for some integer $t > 1$. We may assume $t \le \log^2 a$ for otherwise the theorem is true. Suppose $D = \gcd(a, b)$, $a = D x$ and $b = D y$ with $\gcd(x, y) = 1$. Then
\[
t x^2 (D^2 x^2 + l) = y^2 (D^2 y^2 + l) \; \text{ or } \; D^2 (y^4 - t x^4) = l (t x^2 - y^2) 
\]
after some rearrangement. For $x$ and $y$ sufficiently large in terms of $l$, we must have $t x^2 > y^2$, $0.99 t D^2 x^4 < 1.01 D^2 y^4$ and $1.01 t D^2 x^4 > 0.99 D^2 y^4$. Hence,
\begin{equation} \label{size}
0.9 \sqrt[4]{t} x < y < 1.1 \sqrt[4]{t} x < \sqrt{t} x.
\end{equation}
Moreover, $D^2 \mid l (t x^2 - y^2)$ and
\begin{equation} \label{eq0}
y^4 - t x^4 = s
\end{equation}
with $s = \frac{l (t x^2 - y^2)}{D^2}$ an integer. Note that $s \neq 0$ for otherwise $t x^2 - y^2 = 0 = y^4 - t x^4$ would imply $y^2 = t x^2$ and $y^4 = t x^4$ contradicting $t > 1$. Since $\gcd(x, y) = 1$, we have $\gcd(x, y^4 - t x^4) = 1$, and
\[
\gcd(t x^2 - y^2, y^4 - t x^4) \mid \gcd(t^2 x^4 - y^4, y^4 - t x^4) = \gcd((t^2 - t) x^4, y^4 - t x^4) = \gcd(t^2 - t, y^4 - t x^4).
\]
Thus, $s = \gcd(s, y^4 - t x^4)$ divides $l t (t - 1)$. From $D^2 \mid l (t x^2 - y^2)$, $s \mid l t (t - 1)$, $t \ge 2$ and \eqref{size}, one can deduce that
\begin{equation} \label{sizeD}
\frac{0.3 x}{\sqrt{t}} \le D \le \sqrt{|l| t} x.
\end{equation}
Now, \eqref{eq0} and $a^4 - b^4 = (a - b) (a^3 + a^2 b + a b^2 + b^3)$ imply
\begin{equation} \label{eq1}
\Big| \sqrt[4]{t} - \frac{y}{x} \Big| \le \frac{|l| t^{5/4}}{x^4}.
\end{equation}
By \eqref{BF}, \eqref{Bu}, \eqref{dR4} and \eqref{eq1}, we have
\[
\frac{t^{-16} \exp(- 10^{143} t^{3/2} \log^3 8t )}{x^{4 - 1 / (10^{142} t^{3/2} \log^3 8t)}}  < \Big| \sqrt[4]{t} - \frac{y}{x} \Big| \le \frac{|l| t^{5/4}}{x^4}.
\]
Therefore, $x^{1 / (10^{142} t^{3/2} \log^3 8t)} < |l| t^{69 / 4} \exp( 10^{143} t^{3/2} \log^3 8t )$ which implies
\[
\log x \le 10^{286} \log(|l| + 2) \cdot t^{3} \log^6 8t \; \; \text{ or } \; \; t \ge \frac{(\log x)^{1/3}}{10^{96} \sqrt[3]{\log(|l| + 2)} \, (\log \log x)^{2}}
\]
for $x$ sufficiently large (in terms of $l$). From \eqref{sizeD} and $t \le \log^2 a$, we have
\[
\log Dx \le \log \sqrt{|l| t} x^2 \; \text{ or } \; \log x \ge \frac{1}{2} \log Dx - \log \log Dx - \frac{1}{4} \log |l| \ge \frac{1}{3} \log Dx
\]
for $a = Dx$ sufficiently large (in terms of $l$). Hence,
\[
t \ge \frac{(\log a)^{1/3}}{10^{97} \sqrt[3]{\log(|l| + 2)} \, (\log \log a)^{2}}
\]
and Theorem \ref{thm3} follows from the observation that $b > 0.9 \sqrt[4]{t} a$ via \eqref{size}.
\end{proof}

\section{Proof of Theorem \ref{thm4}}

\begin{proof}
Recall $y^4 = t x^4 + s$ from \eqref{eq0}. Also, we know that $s \mid l t (t - 1)$. Without loss of generality, we may assume $l > 0$ and hence $s > 0$ as the other case can be treated in a similar fashion. Suppose $d = \gcd(t x^4, s)$. Then, we have the integer equation
\[
\frac{y^4}{d} = \frac{t x^4}{d} + \frac{s}{d}
\]
with co-prime entries. By the $abc$-conjecture, we have
\[
\frac{y^4}{d} \ll_\epsilon \Bigl(y x \frac{|l| t (t - 1)}{d} \Bigr)^{1 + \epsilon}. 
\]
Using \eqref{size}, this gives $t^{3/4 - \epsilon / 4} x^{2 - 2 \epsilon} \ll y^{3 - \epsilon} / x^{1 + \epsilon} \ll_{\epsilon, l} t^{2 + 2 \epsilon}$ or $t \gg_{\epsilon, l} x^{8/5 - 5 \epsilon}$. From \eqref{sizeD}, we have $Dx \ll_{l} \sqrt{t} x^2$. Hence
\[
a^{4/7 - 2 \epsilon} = (D x)^{4/7 - 2 \epsilon} \ll_{l} t^{2/7 - \epsilon} x^{8/7 - 4 \epsilon} \ll_{\epsilon, l} t
\]
as $t \gg_{\epsilon, l} x^{8/5 - 5 \epsilon}$ provided $\epsilon < 2/5$. This together with \eqref{size} yields
\[
b \gg_{l} \sqrt[4]{t} a \gg_{\epsilon, l} a^{8/7 - \epsilon}
\]
and, hence, Theorem \ref{thm4}.
\end{proof}


Department of Mathematics \\
Kennesaw State University \\
Marietta, GA 30060 \\
tchan4@kennesaw.edu

\end{document}